\documentclass[11pt,a4paper,reqno]{amsart}
\usepackage{mathrsfs,amssymb}

\newtheorem{theorem}{Theorem}

\newtheorem{corollary}{Corollary}
\newtheorem{definition}{Definition}
\newtheorem{proposition}{Proposition}
\newtheorem{lemma}{Lemma}

\newtheorem{question}{Question}
\newtheorem{example}{Example}

\newcommand{\rinf}{\mathbb{R}^\infty}
\newcommand{\outball}{\overrightarrow{\mathcal{B}}}

\newcommand{\strongball}{\mathcal{B}}

\newcommand{\gr}{\mathcal{R}}
\newcommand{\gl}{\mathcal{L}}
\newcommand{\gh}{\mathcal{H}}

\begin{document}

\title[Groups acting on semimetric spaces]{Groups acting on semimetric
 spaces and quasi-isometries of monoids}

\maketitle

\begin{center}

    ROBERT GRAY

    \medskip

    School of Mathematics \& Statistics,\ University of St Andrews, \\
    St Andrews KY16 9SS, Scotland.

    \medskip

    \texttt{robertg@mcs.st-andrews.ac.uk}

    \bigskip

    \bigskip

    MARK KAMBITES

    \medskip

    School of Mathematics, \ University of Manchester, \\
    Manchester M13 9PL, \ England.

    \medskip

    \texttt{Mark.Kambites@manchester.ac.uk} \\
\end{center}

\begin{abstract}
We study groups acting by length-preserving transformations on spaces
equipped with asymmetric, partially-defined distance functions. We
introduce a natural notion of quasi-isometry for such spaces and
exhibit an extension of the \v{S}varc-Milnor Lemma to this setting. Among
the most natural examples of these spaces are finitely generated monoids
and semigroups
and their Cayley and Sch\"{u}tzenberger graphs; we apply our results to
show a number of important properties of monoids are quasi-isometry
invariants.
\end{abstract}

\section{Introduction}

One of the most exciting and influential developments
of 20th century mathematics was the advent of \textit{geometric
group theory}. A key concept in that subject is \textit{quasi-isometry}: a notion
of equivalence between metric spaces which captures formally the intuitive
idea of two spaces looking the same ``when viewed from far away''.
The \textit{\v{S}varc-Milnor Lemma} (which has been described as the
``fundamental observation of geometric group theory'' \cite{delaHarpe2000})
guarantees that a discrete group which acts in a suitably controlled
way upon a geodesic metric space is quasi-isometric to that space. This
fact establishes a deep connection between geometry and group theory,
which continues to inform our understanding of both subjects.

At the same time, another pervasive theme in modern mathematics has been the
discovery of applications for traditionally pure areas of mathematics;
this trend is typified by the emergence of such fields as algebraic automata
theory \cite{Arbib1968, EilenbergA,EilenbergB} and tropical geometry \cite{Richter-Gebert2005}. Real-world problems
often do not display the high degrees of symmetry and structure enjoyed by 
many of the objects of classical pure mathematics, and this has led pure
mathematicians increasingly to study fundamentally non-rigid, asymmetric
structures, such as semigroups (which arise in algebraic automata theory)
and semirings (which arise in tropical geometry). Another key example is
that of spaces with \textit{asymmetric distance functions}; these have
traditionally been viewed as the poor relations of their symmetric
counterparts, but it is readily apparent that they arise with great
frequency in nature, and hence also in applied mathematics, and they now
seem increasingly relevant also in pure mathematics.

The paper explores a connection between these two trends in mathematics, by
studying discrete groups acting by isometries on
\textit{semimetric spaces}, by which we mean spaces equipped
with asymmetric, partially defined distance functions.
It transpires that both the  notion of quasi-isometry and the
\v{S}varc-Milnor lemma itself admit natural and straightforward
extensions to the more general setting of groups acting on semimetric
spaces. One area of pure mathematics where asymmetric distance functions
naturally arise is the theory of finitely generated
semigroups and monoids; indeed, any finitely generated monoid or
semigroup is naturally endowed with the structure of semimetric space.
We show that a number of natural and widely studied properties of semigroups
and monoids are invariant under quasi-isometry. Some of these are obtained
as applications of our results concerning group actions, while others are
proved directly by geometric arguments.

In addition to this introduction, this article comprises eight sections.
In Section~\ref{sec_semimetric} we introduce quasi-isometries
between semimetric spaces, and some of their foundational properties.
 In Section~\ref{sec_groups} we consider groups acting by
isometries on semimetric spaces, establishing an analogue of the
\v{S}varc-Milnor lemma. Section~\ref{sec_graphs} studies
some important examples of semimetric spaces which arise from   
directed graphs, and in particular from finitely generated semigroups
and monoids.
Section~\ref{sec_apps} demonstrates how group
actions on semimetric spaces can be applied in semigroup theory, by
establishing that a number of natural properties of finitely generated
monoids are invariant under quasi-isometry. Finally, Sections~\ref{sec_growth}
and ~\ref{sec_ends} contain direct (without reliance on group actions)
geometric proofs that some further important properties of finitely generated
monoids (namely growth rate and number of ends) are also quasi-isometry
invariants. 

\section{Semimetric and Quasimetric Spaces}\label{sec_semimetric}

In this section we introduce the main objects of our study, namely spaces
equipped with an asymmetric distance function, and prove a number of
foundational results.

We begin by discussing some issues relating to terminology.
There are several natural ways to generalise the notion of a metric space
by weakening the axioms; these have arisen in numerous different contexts
and terminology for them is not standardised. Spaces with asymmetric
distance functions are perhaps most widely called ``quasi-metric'' spaces
(see for example \cite{Kelly1963, Wilson1931}). Unfortunately,
this terminology conflicts fundamentally with the standard language of
geometric group theory, where the prefix ``quasi'' is by convention used to
mean ``up to finite additive and multiplicative distortion''. To complicate
matters still further, we shall actually need to study spaces which are metric
up to finite distortion, that is, ``quasi-metric'' in the sense that group
theorists would expect. Since we place minimal reliance on the existing
theory, we have chosen to develop terminology consistent with that of geometric
group theory.

Throughout this paper we write $\rinf$ for the set $\mathbb{R}^{\geq 0} \cup \lbrace \infty \rbrace$ of
non-negative real numbers with $\infty$ adjoined. We extend the usual
ordering on $\mathbb{R}^{\geq 0}$ to $\rinf$ in the obvious
way, taking in particular $\infty$ to be the infimum value of
the empty set, and the supremum value of any subset of $\mathbb{R}^{\geq 0}$
which is not bounded above. We also extend addition of non-negative
reals to $\rinf$ and multiplication of positive reals to $\rinf \setminus \lbrace 0 \rbrace$
by defining
$$\infty + x = x + \infty = y \infty = \infty y = \infty$$
for all $x \in \rinf$ and $y \in \rinf \setminus \lbrace 0 \rbrace$.

\begin{definition}[Semimetric space] A \emph{semimetric space} is a pair $(X,d)$
where $X$ is a set, and $d : X \times X \to \rinf$  is a function satisfying:
\begin{itemize}
\item[(i)] $d(x,y)=0$ if and only if $x=y$; and
\item[(ii)] $d(x,z) \leq d(x,y) + d(y,z)$;
\end{itemize}
for all $x,y,z \in X$. If, in addition, $d$ satisfies $d(x,y) = d(y,x)$ then we say that $(X,d)$ is a \emph{metric space}.

A point $x_0 \in X$ is called a \textit{basepoint} for the space $X$ if
$d(x_0, y) \neq \infty$ for all $y \in X$. The space is called
\textit{strongly connected} if every point is a basepoint, that is, if
no two points are at distance $\infty$.
\end{definition}

\begin{definition}[Isometric embeddings]
A map $f : X \to X'$ between semimetric spaces is called
an \textit{isometric embedding} if
$$d(f(x),f(y)) = d(x,y) \ \mbox{for all} \  x, y \in X,$$
and is called an \textit{isometry} if it is a surjective isometric embedding.
\end{definition}

\begin{definition}[Paths and geodesics]
Let $X$ be a semimetric space, $x, y \in X$ and $n \in \mathbb{R}^{\geq 0}$. 
A \textit{path of length $n$ from $x$ to $y$} is a map $p : [0,n] \to X$
such that $p(0) = x$, $p(n) = y$ and $d(p(a), p(b)) \leq b-a$ for all $0 \leq a \leq b \leq n$.

If $d(x, y) \neq \infty$ then a \textit{geodesic} from $x$ to $y$ is a path
of length $d(x,y)$ from $x$ to $y$. The semimetric space $X$ is called
\emph{geodesic} if for all $x,y \in X$ with $d(x,y) \neq \infty$ there
exists at least one geodesic from $x$ to $y$.
\end{definition}
Notice that if $p : [0, d(x,y)] \to X$ is a geodesic from $x$ to $y$ then
for any $a \in [0, d(x,y)]$ we have $d(x, p(a)) \leq a$ and
$d(p(a), y) \leq d(x,y)-a$. But by the triangle inequality,
$d(x,y) \leq d(x,p(a)) + d(p(a),y)$ and distances are positive, so we
deduce that $d(x,p(a)) = a$ and $d(p(a),y) = d(x,y) - a$.

However, a geodesic \textbf{cannot} be defined as an isometric embedding
of a particular space, in the same way that a geodesic in a conventional
metric space is an isometric embedding of an interval on the real line.
The requirements for a function to be a geodesic places demands on 
distances in only one direction; distance in the other direction may
vary, so images of geodesics of the same length need not be isometric.

\begin{definition}[Distance between sets]
Let $X$ be a semimetric space and $A, B \subseteq X$. Then we define the
distance from $A$ to $B$ to be
$$d(A,B) \ = \ \inf_{a \in A, b \in B} d(a,b).$$
\end{definition}
Note in particular that two sets are at distance $\infty$ if and only if
all their members are at distance $\infty$.

\begin{definition}[Balls]
Let $x_0 \in X$ and let $r$ be a non-negative real number. The \emph{out-ball} of radius $r$ based at $x_0$ is
\[
\overrightarrow{\mathcal{B}}_r(x_0) = \{ y \in X: d(x_0,y) \leq r \}.
\]
Dually, the \emph{in-ball} of radius $r$ is defined by
\[
\overleftarrow{\mathcal{B}}_r(x_0) = \{ y \in X: d(y,x_0) \leq r \},
\]
and the \emph{strong ball} of radius $r$ based at $x_0$ is
\[
\mathcal{B}_r(x_0) = \overrightarrow{\mathcal{B}}_r(x_0) \cap
											\overleftarrow{\mathcal{B}}_r(x_0).
\]
\end{definition}

\begin{definition}[Quasi-dense]
Let $T$ be a subset of a semimetric space $X = (X,d)$, and $0 \leq \mu < \infty$.
We say that $T$ is \emph{$\mu$-quasi-dense} in $X$ if for all $x \in X$
there exists $y \in T$ such
that
\[
\max(d(y,x), d(x,y)) \leq \mu.
\]
That is, if $X$ is covered by strong balls of radius $\mu$ around points in $T$.  
A subset is called \textit{quasi-dense} if it is $\mu$-quasi-dense for
some $0 \leq \mu < \infty$.
\end{definition}

\begin{definition}[Quasi-isometric embedding and quasi-isometry]
Let $f:(X,d) \rightarrow (X',d')$ be a map between semimetric spaces, and
$1 \leq \lambda < \infty$ and $0 < \epsilon < \infty$ be constants. 
We say that $f$ is a \emph{$(\lambda,\epsilon)$-quasi-isometric embedding},
and $X$ \textit{embeds quasi-isometrically in $X'$}, if
\[
\frac{1}{\lambda}\;d(x,y) - \epsilon \leq
d'(f(x), f(y)) \leq \lambda d(x,y) + \epsilon
\]
for all $x,y \in X$.
If in addition the image $f(X)$ of $f$ is $\mu$-quasi-dense with
constant $\mu$ then $f$ is called \textit{$(\lambda,\epsilon,\mu)$-quasi-isometry}
and we say that $X$ and $Y$ are \textit{quasi-isometric}.
\end{definition}

The following straightforward proposition can be proved exactly as in
the case of quasi-isometries of metric spaces (see, for example, \cite[Exercise~10.6]{Ghys1991}).

\begin{proposition}\label{prop_qsiprops}
Quasi-isometric embedding is a reflexive and transitive relation on the
class of semimetric spaces. Quasi-isometry is an equivalence relation
on the class of semimetric spaces.
\end{proposition}

Note that any isometric embedding is a quasi-isometric embedding;
in particular, the inclusion map from a subspace into its containing
space is always a quasi-isometric embedding, and is
a quasi-isometry exactly if the subspace is quasi-dense.

\begin{definition}[Quasi-metric space]\label{def_quasimetric}
Let $1 \leq \lambda < \infty$ and $0 \leq \epsilon < \infty$.
A semimetric space $X$ is called \emph{$(\lambda,\epsilon)$-quasi-metric}
if it is strongly connected and $d(y,x) \leq \lambda d(x,y) + \epsilon$
for all $x,y \in X$. A semimetric space is called \textit{quasi-metric}
if it is $(\lambda,\epsilon)$-quasi-metric for some $\lambda$ and $\epsilon$.
\end{definition}
Notice that if a non-empty semimetric space admits $\lambda$ and $\epsilon$
such that the inequalities in Definition~\ref{def_quasimetric} are satisfied,
then it is strongly connected (and hence quasi-metric) if and only if it has a
basepoint. Intuitively, a quasi-metric space is a semimetric space in which the
distances between vertices do not depend too dramatically on the direction
in which one travels between them. One might expect such a space to behave
rather like a metric space; in fact the following easy
proposition says that it must resemble some \textit{particular}
metric space.

\begin{proposition}\label{quasimetric}
A semimetric space $X$ is quasi-metric if and only if it is
quasi-isometric to a metric space.
\end{proposition}
\begin{proof}
Suppose first that $(X,d)$ is a $(\lambda, \epsilon)$-quasi-metric space.
Define $d' : X \times X \to \rinf$ by
$d'(x,y) = d(x,y) + d(y,x)$.
It is well-known and easy to prove that $(X, d')$ is a metric space.
We claim that the identity map $f:X \rightarrow X$ is a
$(\lambda',\epsilon,0)$-quasi-isometry where
 $\lambda' = \max(\lambda + 1, \frac{\lambda}{\lambda + 1})$ (note that $\lambda' \geq 1$).

Since $f$ is surjective, its image is $0$-quasi-dense.
Now for $x,y \in X$ we have
\begin{eqnarray*}
d'(f(x),f(y)) 	& =    & d'(x,y) \\ & = & d(x,y) + d(y,x) \\
								& \leq & d(x,y) + (\lambda d(x,y) + \epsilon) \\
								& =    & \left( \lambda + 1 \right) d(x,y) + \epsilon \\
								& \leq & \lambda' d(x,y) + \epsilon.
\end{eqnarray*}
Also since $(X,d)$ is $(\lambda,\epsilon)$-quasi-metric we have
\begin{eqnarray*}
d'(f(x),f(y)) & =    & d(x,y) + d(y,x) \\
							& \geq & d(x,y) + (\frac{1}{\lambda}d(x,y) - \epsilon) \\
							& = & \left( \frac{\lambda + 1}{\lambda} \right) d(x,y) - \epsilon \\
							& \geq & \frac{1}{\lambda'} d(x,y) - \epsilon.
\end{eqnarray*}

For the converse, suppose that $f: X \rightarrow X'$ is a
$(\lambda, \epsilon, \mu)$-quasi-isometry where $(X',d')$ is a metric
space.
Let $x,y \in X$ be arbitrary. Then since $(X',d')$ is metric we have:
$$\frac{1}{\lambda} d(x,y) - \epsilon \leq d'(f(x),f(y)) = d'(f(y),f(x)) \leq \lambda d(y,x) + \epsilon$$
which implies
\[
d(x,y) \leq \lambda (\lambda d(y,x) + \epsilon + \epsilon) = \lambda^2 d(y,x) + 2 \lambda \epsilon.
\]
So $X$ is $(\lambda^2, 2 \lambda \epsilon)$-quasi-metric.
\end{proof}

\begin{corollary}\label{corol_quasi-metricity}
Quasi-metricity is a quasi-isometry invariant. 
\end{corollary}

An instructive example of a quasi-metric space which is not metric is a
finitely generated group, equipped with the word metric induced by a
(not necessarily symmetric) finite monoid generating set.
We shall see below (see Theorem~\ref{Svarc-Milnor}) that a
semimetric space which admits an suitably controlled action by a group
must be also quasi-metric.

\section{Groups Acting on Semimetric Spaces}\label{sec_groups}

In this section we consider groups acting by isometries on semimetric spaces and prove an analogue of the \v{S}varc--Milnor Lemma. As we shall see in Section~\ref{sec_apps}, such actions arise naturally in semigroup theory.

Let $G$ be a group acting by isometries on the left of a semimetric space $X$.
We write $gx$ for the image of $x \in X$ under the action of $g \in G$.

\begin{definition}[Cocompact action]
An action by isometries of a group $G$ on a semimetric space $X$ is
called \emph{cocompact} if there is a strong ball $B$ of
finite radius such that $\{ gB: g \in G \}$ covers $X$.
\end{definition}

\begin{lemma}\label{proper}
Let $G$ be a group acting by isometries on a semimetric space $(X,d)$. Then the following are equivalent:
\begin{enumerate}
\item[(i)] for every out-ball $B$ of finite radius the set
$
\{ g \in G: d(B, gB) = 0 \}
$
is finite;
\item[(ii)] for every out-ball $B$ of finite radius the set
$
\{ g \in G: d(gB, B) = 0 \}
$
is finite;
\item[(iii)] for every out-ball $B$ of finite radius the set
$
\{ g \in G: B \cap gB \neq \varnothing \}
$
is finite.
\end{enumerate}
The analogous statement also holds both for in-balls, and for strong balls.
\end{lemma}
\begin{proof}
We shall prove the result for out-balls. The results for in-balls and strong balls may be dealt with using similar arguments.

Let $B = \outball_r(x_0)$ be an out-ball and let $g \in G$. If $B \cap gB \neq \varnothing$ then $d(B,gB)=d(gB,B) = 0$. This shows $(\mathrm{(i)} \Rightarrow \mathrm{(iii)})$ and $(\mathrm{(ii)} \Rightarrow \mathrm{(iii)})$. 
To see $(\mathrm{(iii)} \Rightarrow \mathrm{(i)})$ suppose that $d(B,gB) = 0$. Set $C = \outball_{r+\epsilon} (x_0)$ for some fixed $\epsilon > 0$. Then there exist $x \in B$ and $y \in gB$ with $d(x,y) < \epsilon$ which implies
$
y \in C \cap gB \subseteq C \cap gC
$
and hence $C \cap gC \neq \varnothing$.
Therefore 
\[
\{g \in G: d(B,gB) = 0  \} \subseteq \{ g \in G: C \cap gC \neq \varnothing  \}
\]
which is finite by (iii). The proof that (iii) implies (ii) is similar.
\end{proof}

\begin{definition}[Proper action]
We say that $G$ is acting \emph{outward properly} (respectively \emph{inward properly} or \emph{properly}) on $X$ if one of the equivalent conditions given in Lemma~\ref{proper} holds.
\end{definition}

\begin{lemma}\label{lem_proper}
Let $G$ be a group acting by isometries on a semimetric space $X$. If $G$
is acting outward properly or inward properly then $G$ is acting properly.
\end{lemma}
\begin{proof}
Suppose that $G$ is acting outward properly. Let $B$ be a strong ball of finite radius, say $B = \overrightarrow{\mathcal{B}}_r(x_0) \cap \overleftarrow{\mathcal{B}}_r(x_0)$. Let $C = \overrightarrow{\mathcal{B}}_r(x_0)$. Then
\[
\{ g \in G: B \cap gB \neq \varnothing  \} \subseteq
\{ g \in G: C \cap gC \neq \varnothing \}
\]
which is finite since $G$ is acting outward properly.
\end{proof}

The converse of Lemma~\ref{lem_proper} is not true in general. 

In the usual way we regard a finitely generated group as a metric space via its Cayley graph. The quasi-isometry class of a finitely generated group is well-defined since Cayley graphs with respect to different finite generating sets are quasi-isometric. 
More generally this is true for finitely generated semigroups as we shall see in Section~\ref{sec_graphs}.
  
We now prove the main result of this section which is a \v{S}varc--Milnor lemma for groups acting on
semimetric spaces.

\begin{theorem}\label{Svarc-Milnor}
Let $G$ be a group acting outward properly (or inward properly) and
cocompactly by isometries on a geodesic semimetric space $X$ with basepoint.
Then $G$ is a finitely 
generated group quasi-isometric to $X$. In particular, $X$ is a 
quasi-metric space.
\end{theorem}
\begin{proof}
Since $G$ is acting cocompactly there is a strong ball $B$, based at $x_0$
say, of radius $R$ such that $(gB)_{g \in G}$ covers $X$. 
Now $X$ contains a basepoint ($b$ say), which must lie in $gB$ for some
$g \in G$; since the action is by isometries it follows that $g^{-1} b$ is
a basepoint in $B$, and since $B$ is strongly connected we deduce that
$x_0$ is also a basepoint.
Now let
\[
S = \{ g \in G : d(B,gB) = 0 \}.
\]
By Lemma~\ref{lem_proper}, $G$ is acting properly, so the set $S$ is
finite. Clearly $e \in S$, where $e$ denotes the identity element of
$G$ (but it is \textit{not} necessarily the case that $S$ is closed under the
taking of inverses).

Let $C = \overrightarrow{\mathcal{B}}_{5R}(x_0)$, noting that $B \subseteq C$
and define
\[
Q = \{ gB: d(B, gB) \neq 0 \ \mbox{and} \ d(C,gB)=0 \}.
\]
Note that $Q$ is finite, since it is contained in $\{ gB: d(C,gC)=0 \}$,
which is finite since the action is outward proper. Hence, we may choose
a positive real number $r$ such that $r < R$ and $r < d(B,gB)$ for
every $gB \in Q$.

We claim $r$ has the property that for all
$h \in G$ if $d(B,hB)<r$ then $d(B,hB)=0$. To see this, suppose on the contrary that $d(B,hB)<r$ but that $d(B,hB) \neq 0$. Since   $d(B,hB)<r$ there exist $u \in B$ and $v \in hB$ with $d(u,v)<r$. Since $u \in B$ we have $d(x_0,u) \leq R$. Therefore
\[
d(x_0,v) \leq d(x_0,u) + d(u,v) < R + r \leq 2R < 5R
\]
so that $v \in C = \overrightarrow{\mathcal{B}}_{5R}(x_0)$.
Thus $v \in C \cap hB$ so $d(C,hB)=0$ and $hB \in Q$. By the choice of
$r$ it follows that $r < d(B,hB)$. But this contradicts the assumption
that $d(B,hB)<r$, and so proves the claim.

Now choose a positive real number $l < r$. We claim that $S$ generates $G$ and that for all $g \in G$
\[
d_S(e,g) \leq \frac{1}{l}d(x_0,g x_0) +1.
\]
To see this, let $g \in G$ be arbitrary. Since the semimetric space $X$ is
geodesic and $x_0$ is a basepoint, we may choose points $x_1, x_2, \ldots, x_{k+1} = gx_0$ such that $d(x_i,x_{i+1})=l$
for $0 \leq i < k$ and $d(x_k,x_{k+1}) = l' \leq l$.
Define $g_0 = e$, $g_{k+1} = g$ and for $1 \leq i \leq k$ choose $g_i \in G$
such that $x_i \in g_i B$; such choices are possible since $(gB)_{g \in G}$
covers $X$. Now for $0 \leq i \leq k$ set $s_i = g_i^{-1} g_{i+1}$ and observe
that
$$g = g_0 g_0^{-1} g_1 g_1^{-1} g_2 g_2^{-1} \ldots g_k g_k^{-1} g_{k+1} = e s_0 s_1 \ldots s_k.$$
Since $d(x_i,x_{i+1}) \leq l < r$, $x_i \in g_i B$, and $x_{i+1} \in g_{i+1} B$ it follows that $d(g_i B, g_{i+1} B)<r$ which,
since the group is acting by isometries, yields $d(B, g_i^{-1} g_{i+1} B)<r$.
By the above claim we conclude that $d(B, g_i^{-1} g_{i+1} B)=0$ and thus $s_i = g_i^{-1} g_{i+1} \in S$. This proves that $S$ generates the group $G$.

Now $d_S(e,g) \leq k+1$, since we have written $g$ as a product of $k+1$
generators, and $kl = d(x_0,gx_0) - l'$ where $l' \leq l$. So
\[
d_S(e,g) \leq k+1 = \frac{1}{l} d(x_0,g x_0) + (1-\frac{l'}{l}) \leq
\frac{1}{l}d(x_0,g x_0) +1.
\]

Conversely, an easy inductive argument shows that for all $g \in G$, we
have $d(x_0, g x_0) \leq \lambda d_S(e, g)$ for all $g \in G$, where
$\lambda = \max\{d(x_0,sx_0) \mid s \in S \}$.

Now consider the mapping $f: G \rightarrow X$ defined by $g \mapsto g x_0$. It follows from the observations above that
\[
d_S(g_1,g_2) \leq \frac{1}{l} d(f(g_1), f(g_2)) +1
\]
and also
\[
d(f(g_1),f(g_2)) \leq \lambda d_S(g_1,g_2)
\]
for all $g_1, g_2 \in G$. Moreover, given $x \in X$, since $(\alpha B)_{\alpha \in G}$ covers $X$ we conclude that there exists $h \in G$ with $x \in hB$, and thus
\[
\max(d(f(h),x),d(x,f(h))) \leq R.
\]
Hence $G$ and $X$ are quasi-isometric, and since the Cayley graph of $G$ is
a quasi-metric space, by Corollary~\ref{corol_quasi-metricity}, $X$ is a quasi-metric space.
\end{proof}

As an alternative proof strategy for Theorem~\ref{Svarc-Milnor}, one might
think to first prove directly that $X$ must be quasimetric, before deducing
that $G$ acts properly and cocompactly by isometries on the corresponding
metric space constructed in the proof of Proposition~\ref{quasimetric} and
applying the usual \v{S}varc-Milnor Lemma. However, there seems to be no
easier way to prove that $X$ is quasimetric than by establishing 
Theorem~\ref{Svarc-Milnor}.

\section{Directed Graphs and Semigroups as Semimetric Spaces}\label{sec_graphs}

In this section we consider a particularly important class of semimetric
spaces, namely those which arise from directed graphs.

By a \textit{(directed) graph} $\Gamma$ we mean a set $R$ of \textit{vertices}
together with a set
$E$ of \textit{edges} and two functions $\iota : E \to R$ and
$\tau : E \to R$ which describe respectively the \textit{initial
vertex} (or \textit{source}) and the \textit{terminal vertex} (or \textit{target})
of each edge. Note that this definition permits loops and multiple edges,
and places no cardinality restrictions on the vertex or edge sets.

It is easily verified that the vertex set of a directed graph $\Gamma$ is
an example of a semimetric space, with the distance $d(x,y)$ between vertices
$x$ and $y$ defined to be the infimum number of edges in a directed path from
$x$ to $y$ ($\infty$ if there is no path). For our purposes, it will also
be convenient to regard a directed graph as a \textit{geodesic}
space, to facilitate the application of \textit{continuous} arguments. To
this end, given a directed graph $\Gamma$ we define a new semimetric
space $\Gamma^*$ with point set $\Gamma^* = R \cup (E \times (0, 1))$ and metric defined
as follows. If $x,y \in R$ are vertices in $\Gamma$ then $d(x,y)$ is the
shortest length of a path from $x$ to $y$ in $\Gamma$.
Otherwise we define
\begin{eqnarray*}
d((e,\mu),y) & = & (1-\mu) + d(\tau(e),y) \\
d(x,(e,\mu)) & = & d(x,\iota(e)) + \mu
\end{eqnarray*}
\[
d((e,\mu),(f,\nu)) = \begin{cases}
\nu - \mu & \mbox{if} \ e=f \ \mbox{and} \ \nu \geq \mu \\
d(\tau(e),\iota(f)) + (1-\mu) + \nu & \mbox{otherwise}.
\end{cases}
\]
The following elementary relation between $\Gamma$ and $\Gamma^*$ is now
easily established.
\begin{proposition}\label{prop_graphembed}
Let $\Gamma$ be a directed graph. Then $\Gamma^*$ is a geodesic semimetric space and the
inclusion $\Gamma \subseteq \Gamma^*$ is an isometric embedding.
\end{proposition}

Notice that the inclusion of $\Gamma$ into $\Gamma^*$ is \textbf{not} in
general a quasi-isometry. Since a typical point on a edge is
``near'' each of its endpoints in only \textbf{one} direction, there is
no reason to suppose it lies in a strong ball about any vertex, let alone
one of uniformly bounded radius. In fact, it is easy to show that the
inclusion of $\Gamma$ into $\Gamma^*$ is a quasi-isometry if and
only if $\Gamma$ (and hence $\Gamma^*$ by Corollary~\ref{corol_quasi-metricity}) is quasi-metric.

We now turn our attention to some semimetric spaces which arise naturally
in semigroup theory. Let $S$ be a semigroup generated by a finite subset $A$.
Then $S$ is naturally endowed with the structure of a semimetric space,
with distance function defined by
\[
d_A (x,y) = \inf\{ |w| : w \in A^*, \ xw = y  \}
\]
where $A^*$ denotes the free monoid over $A$, and $|w|$ is the length of the word $w$. 
Note that $d_A (x,y) = \infty$ is possible since a semigroup can have proper right ideals. 
For $s \in S = \langle A \rangle$ we use $l_A(s)$ to denote the minimal length of a word over $A^+$ that represents the element $s$. 

A slightly more sophisticated semimetric space may be obtained by
considering the {(right) Cayley graph} of $S$ with respect to $A$ which is the
edge-labelled directed graph whose vertices are the elements of $S$ and with
a directed edge from $x \in S$ to $y \in S$ labelled by $a \in A$ if and only
if $xa = y$ in $S$. We denote by $\Gamma(S,A)$ the geodesic semimetric space
obtained by applying the $*$-operation to this graph.

It follows from Proposition~\ref{prop_graphembed} that the natural inclusion
of $S$ into $\Gamma(S,A)$ is an isometric embedding, but in general there is
no reason to suppose it is a quasi-isometry.
Moreover, the semimetric space $\Gamma(S,A)$ carries more information than
the semimetric space $(S,d_A)$, since it contains information about
edge-multiplicities. So in general, it is impossible to
reconstruct $\Gamma(S,A)$ from $(S,d_A)$, while conversely $(S,d_A)$
can be obtained from $\Gamma(S,A)$ just by restricting to the set of vertices.

There are also the obvious dual notion of left Cayley graph which has the same vertex set but different semimetric. 
For finitely generated groups the corresponding left and right versions of each of these spaces are isometric so nothing is lost by always working just with right Cayley graphs. For finitely generated semigroups the right and left Cayley graphs are not, in general, isometric. In fact, they may not even be quasi-isometric (this is easily seen for example by taking a semigroup with a different number of $\gr$-classes than $\gl$-classes, in the sense defined below). From now on, unless otherwise stated,
by Cayley graph we shall always mean right Cayley graph.

\begin{proposition}\label{Independence}
Let $A$ and $B$ be finite generating sets for a semigroup $S$. Then $(S,d_A)$ is quasi-isometric to $(S,d_B)$.
\end{proposition}
\begin{proof}
Let $f$ be the identity mapping on $S$, viewed as a map from $(S,d_B)$ to
$(S,d_A)$. Since $f$ is surjective its image is quasi-dense in
$(S,d_A)$. Set $\lambda_1 = \max\{ l_A(b) : b \in B  \}$ and
$\lambda_2 = \max\{ l_B(a) : a \in A  \}$. For all $x,y \in S$ it is
easy to check by induction on $d_B(x,y)$ that $d_A(f(x),f(y)) \leq \lambda_1 d_B(x,y)$.  Similarly $d_B(x,y) \leq \lambda_2 d_A(f(x),f(y))$.
We conclude that $f$ is a $(\max\{ \lambda_1, \lambda_2 \}, 0, 0)$-quasi-isometry.
\end{proof}

However, in contrast to groups, if $A$ and $B$ are finite generating sets
for a semigroup $S$ then $\Gamma(S,A)$ and $\Gamma(S,B)$ need not be
quasi-isometric (see Example~\ref{Example_Cayley_Graph} below). 

\begin{definition}\label{def_qsi}
Let $S$ and $T$ be a finitely generated semigroups with finite generating sets $A$ and $B$, respectively. We say that the semigroups
$S$ and $T$ are (right) quasi-isometric if $(S,d_{A})$ and $(T,d_{B})$ are quasi-isometric.
\end{definition}
Again, there is a dual notion of two semigroups being left quasi-isometric. Here we shall work only with right quasi-isometries between semigroups and by quasi-isometric we shall always mean right quasi-isometric.
Note that as a result of the dependence on choice of generating set described above, we cannot use the right Cayley graph $\Gamma(S,A)$ in the above definition.

For a semigroup $S$ we use $S^1$ to denote the monoid $S \cup \{ 1 \}$ where $1$ is an adjoined identity, assumed to be disjoint from $S$. 
Let $S$ and $T$ be finitely generated semigroups. It is easy to see that $S$ and $T$ are quasi-isometric if and only if $S^1$ and $T^1$ are quasi-isometric. (For the less trivial of the two directions, one just has to observe that any quasi-isometry $f:S^1 \rightarrow T^1$ must map $1_S$ to $1_T$ since these are the unique basepoints in the respective Cayley graphs.) In our discussions below there are various situations where we shall find it convenient to work with monoids rather than semigroups, and from this observation we see that no generality is lost in doing so. 

Associated with any semimetric space $X$ is a natural preorder $\lesssim$ relation given by $x \lesssim y$ if and only if $d(y,x) < \infty$. Let $\sim$ denote the equivalence relation given by $x \sim y$ if and only if
$x \lesssim y$ and $y \lesssim x$. We call the $\sim$-classes the \emph{strongly connected components} of $X$. Given a semimetric space $X$ let $X / \sim$ denote the poset of equivalence classes of strongly connected components of $X$. The following proposition is an immediate consequence of the
definition of quasi-isometry.

\begin{proposition}\label{prop_components}
Let $f: X \rightarrow Y$ be a quasi-isometry of semimetric spaces. Then $f$ maps each of the  $\sim$-classes of $X$ quasi-isometrically into a $\sim$-class of $Y$, and this induces an isomorphism of partially ordered sets $(X / \sim) \rightarrow (Y / \sim)$.
\end{proposition}

\begin{example} \label{Example_Cayley_Graph}
Let $S$ be the two element semigroup $\{ a, 0 \}$ where $a^2=0$ and $a0 = 0a = 0^2 = 0$. Then $A = \{ a \}$ and $B = \{ a,0 \}$ are both generating sets for $S$ but $\Gamma(S,A)$ and $\Gamma(S,B)$ are not quasi-isometric. Indeed, the poset associated to $\Gamma(S,A)$ has the property that any pair of points is comparable, while this is not the case in the poset associated with $\Gamma(S,B)$. Hence by Proposition~\ref{prop_components} the spaces $\Gamma(S,A)$ and $\Gamma(S,B)$ are not quasi-isometric. 
\end{example}

In particular it follows from Proposition~\ref{prop_components} that if two semigroups $S$ and $T$ are
quasi-isometric then the partially ordered sets $S / \gr$ and $T / \gr$ must be isomorphic (where $\gr$ denotes Green's $\gr$-relation defined below).

A central question in geometric group theory is that of which algebraic
properties are invariant under quasi-isometry. Such properties are called
\emph{geometric} and well-known examples include finiteness, the number of
ends, having a free subgroup of finite index, being
finitely presented, 
being hyperbolic, being automatic, being amenable, being accessible, the
type of growth (see \cite[p115, Section 50]{delaHarpe2000} and the references therein), 
having an abelian subgroup
of finite index (see \cite{Bridson1996} and \cite{Pansu1983}), having a nilpotent subgroup of finite index (see \cite{Gromov1981}), being finitely presented with solvable word problem, and satisfying the homological finiteness condition $F_n$
or the condition $FP_n$ (see \cite{Alonso1990} and \cite{Alonso1994}).

It is natural to ask which properties are quasi-isometry invariants
of semigroups. Certain properties, like being finite, or having finitely many
right ideals, are clearly quasi-isometry invariants (the latter
observation follows from Proposition~\ref{prop_components}). 
We shall see below that several important properties of finitely generated semigroups 
are invariant under quasi-isometry. 

\section{Sch\"{u}tzenberger Groups and Graphs} \label{sec_apps}

In this section we demonstrate how the theory developed in the preceding
section can be applied to some problems in the theory of finitely generated
semigroups and monoids. For any undefined concepts from semigroup theory we 
refer the reader to \cite{Howie1995}. 

The notions of Sch\"{u}tzenberger graph, and Sch\"{u}tzenberger group, lie
at the heart of recent developments in geometric approaches in semigroup
theory; see Steinberg \cite{Steinberg2001, Steinberg2003}. In \cite{Ruskuc2000} it was shown that under a certain finiteness assumption (see below) a Sch\"{u}tzenberger group of a finitely generated monoid will be finitely generated. In fact, as pointed out in \cite{Ruskuc2000}, this result also follows from Sch\"{u}tzenberger's original work \cite{Schutzenberger1957, Schutzenberger1958}. In \cite{Steinberg2003} a topological proof of the same result was given for the special case of maximal subgroups of inverse semigroups. Here by considering the natural action of the Sch\"{u}tzenberger group on its Sch\"{u}tzenberger graph, applying the
\v{S}varc--Milnor lemma of Section~\ref{sec_groups} we shall obtain an alternative proof of this result. 
At the same time we recover information relating the geometry of the Sch\"{u}tzenberger graph to that of the Sch\"{u}tzenberger group. This relationship is used below when we consider quasi-isometry invariants of semigroups.

First we must introduce some ideas and terminology from semigroup theory. Green's relations were introduced in \cite{Green1951}, and ever since have played a fundamental role in the structure theory of semigroups. We give a brief overview of the theory here, for more details we refer
the reader to \cite{Howie1995}.

On any monoid $M$, we may define a pre-order $\leq_\gr$ by $x \leq_\gr y$
if and only if $xM \subseteq yM$. The relation $\gr$ is defined to be the
least equivalence relation containing $\leq_\gr$, so $x \gr y$ if and only
if $x$ and $y$ generate the same principal right ideal. If $M$ is generated
by $A$ then the $\gr$-classes are the strongly connected components of
$(M,d_A)$. A pre-order
$\leq_\gl$ and equivalence relation $\gl$ can be defined in the obvious
left-right dual way, and the intersection $\gr \cap \gl$ (which is also
an equivalence relation) is denoted $\gh$.

The importance of the $\gh$-relation becomes apparent when considering the
maximal subgroups of a monoid; those $\gh$-classes which contain
idempotents are exactly the maximal subgroups of the containing monoid.
It is possible to associate a group to
any $\gh$-class of a monoid, which is called the \emph{Sch\"{u}tzenberger
group} of the $\gh$-class. If the $\gh$-class happens to be a subgroup then
this group will be isomorphic to the Sch\"{u}tzenberger group of the
$\gh$-class, so the notion of Sch\"{u}tzenberger group generalises that
of maximal subgroup. The (left) Sch\"{u}tzenberger group is obtained by
taking the action of the setwise stabiliser of $H$ on $H$, under left
multiplication by elements of the monoid, and making it faithful. That
is, given an arbitrary $\gh$-class $H$ of $M$, let $\mathrm{Stab}(H) = \{ s \in S : sH=H \}$ denote the \emph{(left) stabilizer} of $H$ in $S$.
Then define an equivalence
$\sigma=\sigma(H)$ on the stabilizer by $(x,y) \in \sigma$ if and only if $xh = yh$ for all $h\in H$. It is straightforward to verify that $\sigma$ is a congruence, and that $\mathcal{G}(H)=\mathrm{Stab}(H) / \sigma$ is a group,
called the \emph{left Sch\"{u}tzenberger group} of $H$.
One can also define the right Sch\"{u}tzenberger group of $H$ in the natural
way, and it turns out that the left and right Sch\"{u}tzenberger groups are isomorphic to one another.
For information about the basic properties of Sch\"{u}tzenberger groups we
refer the reader to \cite[Section~2.3]{Lallement1979}. In particular we recall
here that the orbits of the action of $\mathcal{G}$ on $R$ are precisely
the $\gh$-classes of $S$ contained in $R$, and that the action of
$\mathcal{G}$ on such an $\gh$-class is \textit{semiregular} in the sense of
permutation group theory, that is, that only the identity has a fixed point.

The \emph{Sch\"{u}tzenberger graph} $\Gamma(R,A)$ of $R$, with respect to
$A$, is the strongly connected component of $h \in H$ in $\Gamma(M,A)$.
It is easily seen to consist of those vertices which are elements of $R$,
together with edges connecting them, and so can be obtained by beginning
with a directed graph $\Delta$ with vertex set $R$ and a directed labelled edge from $x$ to $y$ labelled by $a \in A$ if and only if $xa = y$, and then setting $\Gamma(R,A) = \Delta^*$ (using the notation from Section~3). From its construction it is clear that for any generating set $A$ of $M$, $\Gamma(R,A)$ is a connected geodesic semimetric space. 

Now the group $\mathcal{G}(H)$ acts naturally on $R$ via
$(s / \sigma) \cdot r = sr$. This action extends naturally to an action
by isometries of $\mathcal{G}(H)$ on the Sch\"{u}tzenberger graph
$\Gamma(R,A)$. Thus, the group $\mathcal{G}(H)$ acts by isometries on
the strongly connected geodesic semimetric space $\Gamma(R,A)$.

\begin{theorem}\label{thm_schutzact}
Let $M$ be a monoid generated by a finite set $A$, let $H$ be an $\gh$-class of $M$,
let $G$ be the  Sch\"{u}tzenberger group of $H$, and let $\Gamma(R,A)$ denote the
Sch\"{u}tzenberger graph of the $\gr$-class $R$ containing $H$. Then the left
translation action of $G$ on $\Gamma(R,A)$ is outward proper and by isometries. The
action is cocompact if and only if $R$ contains only finitely many
$\gh$-classes.
\end{theorem}
\begin{proof}
To show that the action is by isometries, it clearly suffices to show
that the left translation action of $G$ on $R$ is by isometries. If
$x, y \in R$ are such that $d(x,y) = n$ then there exist $a_1, \dots, a_n \in A$ such that
$x a_1 \dots a_n = y$. But then
$$(s / \sigma) \cdot x a_1 \dots a_n = s x a_1 \dots a_n = s y = (s / \sigma) \cdot y$$
so that $d((s / \sigma) \cdot x, (s / \sigma) \cdot y) \leq n = d(x,y)$. A similar
argument using $(s / \sigma)^{-1}$ shows that
$d((s / \sigma) \cdot x, (s / \sigma) \cdot y) \geq d(x,y)$.

To prove that the action is outward proper we must show that
for each $x_0 \in R$ and $0 \leq \epsilon < \infty$ the set
\[
Q = \{ g \in G \ | \  g \outball_\epsilon(x_0) \cap \outball_\epsilon(x_0) \neq \varnothing \}
\]
is finite. Since $\Gamma(R,A)$ is strongly connected, it will clearly suffice
to fix a basepoint $x_0 = h \in H$ and prove the claim for that $x_0$ and
every $0 \leq \epsilon < \infty$.

Set $B = \outball_\epsilon(x_0)$ noting that $B$ contains only
finitely many vertices of $\Gamma(R,A)$ since the out-degree of every
vertex is bounded above by $|A|$. Let $\{x_1, \ldots, x_n\}$ be the set of vertices in $B$.
Let $g \in Q$. By definition
there exists $b \in B$ such that $gb \in B$. First suppose that $b$ is a vertex, so that $b=x_i$ for some $i$. Since under the action of $G$ on $\Gamma(R,A)$ vertices are mapped to vertices,
we have $gb = x_j$ for some $j$. But since $G$ acts on $R$ with trivial point stabilizers, it follows that $g$ is uniquely determined by the pair $(x_i,x_j)$ so must be one of a fixed set of $n^2$ group elements.

If $b$ is not a vertex, say $b = (e,\mu)$ then consider the strong ball of radius $\epsilon + \mu$ based at $\iota(e)$ and apply the argument of the previous paragraph. This completes the proof that the action is
outward proper.

For the second part of the theorem, suppose that $R$ is a union of finitely
many $\gh$-classes. Let $x_0 = h \in H$. Since $(R,d_A)$ is strongly connected
and $R$ is a union of finitely
many $\gh$-classes it follows that there exists $\lambda \geq 0$ such that the strong ball $\mathcal{B}_{\lambda}(x_0)$ intersects every $\gh$-class in $R$. 
Since the orbits of $G$ on $R$ are the $\gh$-classes in $R$ it follows that the translates $(g \strongball_\lambda(x_0) )_{g \in G}$ cover all the vertices of $\Gamma(R,A)$. Since every point of $\Gamma(R,A)$ is within
distance $1$ of a vertex, it follows that with $\epsilon = \lambda + 1$ the translates $(g \strongball_\epsilon(x_0))_{g \in G}$ cover $\Gamma(R,A)$. Hence the action is cocompact.

For the converse, suppose that the action is cocompact. Then there is a strong ball $B$  of finite radius whose translates
$(gB)_{g \in G}$ cover $\Gamma(R,A)$. Since $B$ has finite radius, it only intersects finitely many of the $\gh$-classes in $R$. As the $\gh$-classes of $R$ are the orbits under the action of $G$ on $R$, it follows that any vertex $h \in gB$ belongs to one of the finitely many $\gh$-classes that $B$ intersects. Since $(gB)_{g \in G}$ covers $\Gamma(R,A)$ we conclude that $B$ intersects every $\gh$-class in $R$, and thus $R$ is a union of finitely many $\gh$-classes.
\end{proof}

Theorems about finitely generated semigroups are often proved by technical,
combinatorial means which, while convincing, yield relatively little insight
into \textit{why} the results hold.
A case in point is a theorem of Ruskuc \cite{Ruskuc2000}, originally proved
using a Reidemeister-Schreier argument, which states that in a finitely generated monoid,
if an $\gh$-class $H$ lies in an $\gr$-class containing only finitely many $\gh$-classes, then the Sch\"{u}tzenberger group of $H$ is finitely generated. 
Steinberg \cite{Steinberg2003}
has shown that this fact has a geometric interpretation in the special
case of inverse semigroups. The following stronger statement
arises as an immediate corollary of Theorems~\ref{Svarc-Milnor} and
\ref{thm_schutzact}; as well as being an alternative proof, we would
argue that it provides also a more satisfactory \textit{explanation}
for this phenomenon.

\begin{theorem}\label{quasiisomschutz}
Let $M$ be a monoid generated by a finite set $A$, and let $H$ be an
$\gh$-class of $M$. If the $\gr$-class $R$ that contains $H$ has only
finitely many $\gh$-classes then the Sch\"{u}tzenberger group $G$ of
$H$ is finitely generated and quasi-isometric to the Sch\"{u}tzenberger
graph $\Gamma(R,A)$.
\end{theorem}

It follows from the comment after Proposition~\ref{prop_graphembed} that under the assumptions of Theorem~\ref{quasiisomschutz}, $\mathcal{G}(H)$, $\Gamma(R,A)$ and $(R,d_A)$ are all quasi-isometric to one another.
Therefore by the argument of Proposition~\ref{Independence}, if $R$ has only finitely many $\gh$-classes, and $A$ and $B$ are any two finite generating sets for $M$, then the semimetric spaces $\Gamma(R,A)$ and $\Gamma(R,B)$ are quasi-isometric.

Theorem~\ref{quasiisomschutz} is an example of a situation where one may apply the \v{S}varc--Milnor lemma of Section~\ref{sec_groups} and thus avoid having to prove
a Reidemeister-Schreier rewriting result. We mention also two other situations where Theorem~\ref{Svarc-Milnor} may be similarly applied.

In \cite[Lemma~6.3]{Steinberg2009} the author proves (in the terminology defined in that paper) that if $S$ is a finitely generated inverse semigroup acting non-degenerately on a locally compact Hausdorff space $X$ and the orbit $\mathcal{O}$ of a point $x \in X$ is finite, then the isotropy group $G_x$ (defined in \cite[Definition~6.1]{Steinberg2009}) is itself finitely generated. As the author points out, the isotropy group $G_x$ acts naturally on the right of $L_x = d^{-1}(x) = \{ [s,x] : s \in S \}$. The set $L_x$ plays the role of the $\gl$-class in this setting. The set $L_x$ is the vertex set of a digraph (which is the analogue of the left Sch\"{u}tzenberger graph) where there is a directed edge from $[s_1,x]$ to $[s_2,x]$ labelled by $[a,s_1 x]$ when $s_1, s_2 \in S$ satisfy $as_1 = s_2$. The action of the isotropy group $G_x$ on $L_x$ extends naturally to an action by isometries on this digraph, viewed as a semimetric space, and when the orbit $\mathcal{O}$ is finite the action is cocompact. Then as in Theorem~\ref{quasiisomschutz} above, one may deduce that the isotropy group $G_x$ is finitely generated as an application of Theorem~\ref{Svarc-Milnor}. 

In \cite{Gray2008} Green's relations and Sch\"{u}tzenberger groups are considered, but taken relative to a subsemigroup of a semigroup. Given a semigroup $S$ and a subsemigroup $T$ we write 
\[
u \mathcal{R}^T v \Leftrightarrow uT^1 = vT^1, \quad u \mathcal{L}^T v \Leftrightarrow T^1 u = T^1 v
\] 
and $\mathcal{H}^T = \mathcal{R}^T \cap \mathcal{L}^T$. In \cite{Gray2008} the \emph{Green index} of $T$ in $S$ is defined to be one more than the number of $\mathcal{H}^T$-classes in $S \setminus T$. There are natural corresponding notions of $T$-relative Sch\"{u}tzenberger groups and Sch\"{u}tzenberger graphs. In this context, if $T$ is finitely generated by a set $A$, $R$ is an $\mathcal{R}^T$-class of $S$, and  $H$ is an $\mathcal{H}^T$-class of $S$ contained in $R$, then the (left) $T$-relative Sch\"{u}tzenberger group $\Gamma(H)$ of $H$ acts by isometries on the (right) $T$-relative Sch\"{u}tzenberger graph $\Gamma(R,A)$, and this action is proper by \cite[Proposition~5]{Gray2008}. When $R$ is a union of finitely many $\mathcal{H}^T$-classes this action is cocompact and, just as in Theorem~\ref{quasiisomschutz} above, it follows that $\Gamma(H)$ is finitely generated and quasi-isometric to $\Gamma(R,A)$. In particular this shows that if $T$ is finitely generated and has finite Green index in $S$, then all of the relative Sch\"{u}tzenberger groups of $\mathcal{H}^T$-classes in $S \setminus T$ are finitely generated. 

In general Sch\"{u}tzenberger graphs can be very far away from being quasi-isometric to groups, as the
following straightforward  examples show.

\begin{example}
Let $T_{\mathbb{Z}}$ denote the full transformation monoid on the set $\mathbb{Z}$ of integers. The elements of $T_{\mathbb{Z}}$ are the maps from $\mathbb{Z}$ to itself, and multiplication is given by usual composition of maps, where we view maps as acting on the right, and compose from left to right. Let $\alpha$ and $\alpha^{-1}$ denote the infinite cycles given by 
\[
n \alpha = n+1, \quad n \alpha^{-1} = n-1 \quad \mbox{for} \ n \in \mathbb{Z},
\]
and for $i \in \mathbb{Z}$ let $\gamma_i$ denote the constant mapping with image $i$. Set $A = \{ \alpha, \alpha^{-1}, \gamma_0 \}$ and define
\[
S = \langle A \rangle = \{ \alpha^{n} : n \in \mathbb{Z} \} \cup \{ \gamma_i : i \in \mathbb{Z} \}. 
\] 
The $\gr$-class of $\gamma_0$ is the set of all constant maps $R = \{ \gamma_i : i \in \mathbb{Z} \}$. 
We claim that the Sch\"{u}tzenberger graph $\Gamma(R,A)$ is not quasi-metric. Indeed, for all $i \in \mathbb{Z}$, $d(\gamma_n, \gamma_0) = 1$ (since $\gamma_n \gamma_0 = \gamma_0$) while $d(\gamma_0, \gamma_n) = |n|$ (the unique shortest directed path from $\gamma_0$ to $\gamma_n$ is the one labelled by the word $\alpha^n$). 

Also, in this example the Sch\"{u}tzenberger group of $R$ is the trivial group while the Sch\"{u}tzenberger graph is infinite. 
\end{example}

\begin{example}\label{ex_bicyclic}
Consider the bicyclic monoid $B$ defined by the finite presentation $\langle b,c \ | \ bc=1 \rangle$, with respect to its usual generating set $A = \{ b,c \}$. The Sch\"{u}tzenberger graph $\Gamma(R_1,A)$ containing the identity element is a one-way infinite line with directed edges in both directions between adjacent vertices of the line. By the indegree of a vertex $v$ of the digraph $\Gamma(R_1,A)$ we mean the number of vertices $w$ such that there is a directed edge from $w$ to $v$. Dually we define the outdegree. In $\Gamma(R_1,A)$ all the vertices have finite indegree and outdegree, there is a unique vertex (namely the identity of $B$) with indegree and outdegree $1$ while all other vertices have indegree and outdegree $2$. The space $\Gamma(R_1,A)$ is quasi-metric, but it is not quasi-isometric to the corresponding Sch\"{u}tzenberger group which is trivial.
\end{example}

Theorem~\ref{quasiisomschutz} has
some consequences regarding quasi-isometry invariants of finitely generated monoids.

\begin{theorem}\label{finite_presentability_special_case}
For finitely generated monoids with finitely many left and right ideals,
finite presentability is a quasi-isometry invariant.
\end{theorem}
\begin{proof}
Let $M$, $N$ be quasi-isometric finitely generated monoids, each with finitely many left and right ideals, and with finite generating sets $A$, $B$ respectively. Suppose that $M$ is finitely presented. We want to show that $N$ is finitely presented. Let $f: (N,d_B) \rightarrow (M,d_A)$ be a quasi-isometry. Let $R$ be an $\gr$-class of $N$. By Proposition~\ref{prop_components}, $f(R)$ is contained in an $\gr$-class, $K$ say, of $M$. Let $H \subseteq R$ be an $\gh$-class of $N$, and let $U \subseteq K$ be an $\gh$-class of $M$. Then by Theorem~\ref{quasiisomschutz} and transitivity of quasi-isometry we see that the semimetric spaces $\mathcal{G}(H)$, $(R,d_A)$, $(K,d_B)$ and $\mathcal{G}(U)$ are all quasi-isometric to one another. Thus the groups $\mathcal{G}(H)$ and $\mathcal{G}(U)$ are quasi-isometric. By assumption, $M$ is finitely presented which by \cite[Theorem~1.1]{Ruskuc2000} implies that $\mathcal{G}(U)$ is finitely presented. Since
having a finite presentation is a quasi-isometry invariant for groups (see \cite[p115, Section 50]{delaHarpe2000}) it follows that $\mathcal{G}(H)$ is finitely presented. As $H$ was an arbitrary $\gh$-class of $N$ it follows that all Sch\"{u}tzenberger groups of $N$  are finitely presented which, along with the fact that $N$ has finitely many left and right ideals, by \cite[Theorem~1.1]{Ruskuc2000} implies that $N$ is finitely presented.
\end{proof}

For certain important classes of semigroups, finite generation is sufficient
to guarantee that there are only finitely many $\gl$- and $\gr$-classes,
and so in these cases finite presentability is a quasi-isometry invariant.

\begin{corollary}
Finite presentability is a 
quasi-isometry invariant for finitely generated Clifford monoids, and for 
finitely generated completely simple and completely $0$-simple semigroups.
\end{corollary}

Note that although completely ($0$-)simple semigroups are not in general monoids, by the comment made after Definition~\ref{def_qsi}, Theorem~\ref{finite_presentability_special_case} can still be applied to them.  

\begin{question}
Is finite presentability a quasi-isometry invariant of finitely generated semigroups in general?
\end{question}

A natural first step towards answering this question would be to first establish
whether finite presentability is an \emph{isometry} invariant of finitely generated semigroups.
We finish the section with a few variants of Theorem~\ref{finite_presentability_special_case}.

\begin{theorem}\label{thm_sol_wp}
For finitely generated monoids with 
finitely many left and right ideals, the property of being finite 
presented with solvable word problem is a quasi-isometry invariant. 
\end{theorem}
\begin{proof}
Although not explicitly stated there, it is an 
easy consequence of the proof of \cite[Theorem~3.2]{Ruskuc2000} that for a monoid 
$M$ with finitely many left and right ideals, $M$ is finitely presented 
with solvable word problem if and only if all of its Sch\"{u}tzenberger 
groups are finitely presented with solvable word problem. Now the result 
follows by an argument identical to the one given in the proof of 
Theorem~\ref{finite_presentability_special_case}, along with the result of \cite{Alonso1990} stating that being finitely presented with solvable word problem is a quasi-isometry invariant of groups. 
\end{proof}

Recall that a semigroup is \textit{regular} if every $\gr$-class contains an idempotent.
In \cite{Alonso1994} it was shown that the homological finiteness condition $FP_n$ is a quasi-isometry invariant of finitely generated groups. Along with the result in \cite{GrayMalheiro} for regular monoids with finitely many left and right ideals, and the result of \cite{Cremanns1996} saying that for finitely presented groups the homotopical finiteness condition \emph{finite derivation type} and $FP_3$ are equivalent, using the same argument as in the proof of Theorem~\ref{finite_presentability_special_case} we obtain the following. For the definition of finite derivation type and more on its importance in the theory of string-rewriting systems and connections to the theory of diagram groups see \cite{Guba1997, Otto1997}. 

\begin{theorem}\label{thm_fdt}
For finitely generated regular monoids with finitely many left and right ideals having finite derivation type is a quasi-isometry invariant.
\end{theorem}

Just as for Theorem~\ref{finite_presentability_special_case} each of Theorems~\ref{thm_sol_wp} and \ref{thm_fdt} apply to finitely generated Clifford monoids and completely $0$-simple semigroups.

\section{Growth}\label{sec_growth}

In this section we study the relationship between quasi-isometry and
the notion of \textit{growth} in semigroups and monoids. Recall that 
a (discrete) \emph{growth function} is a monotone non-decreasing function from $\mathbb{N}$ to $\mathbb{N}$.
Growth functions for finitely generated groups were introduced independently
by Milnor \cite{Milnor1968} and \v{S}varc \cite{Svarc1955}; since then growth of both
groups and monoids has become a subject of extensive study (see for example
  \cite{Bergman1978, Grigorchuk1988, Shneerson2001,  Shneerson2005, Shneerson2008}).
For growth functions $\alpha_1, \alpha_2$ we write $\alpha_1 \preccurlyeq \alpha_2$ if there exist natural numbers $k_1, k_2 \geq 1$ such that
$\alpha_1(t) \leq k_1 \alpha_2(k_2 t)$
for all $t \in \mathbb{N}$. We define an equivalence relation on growth functions
by $\alpha_1 \sim \alpha_2$ if and only if $\alpha_1 \preccurlyeq \alpha_2$ and $\alpha_2 \preccurlyeq \alpha_1$.
The $\sim$-class $[\alpha]$ of a growth function $\alpha$ is called the
\emph{growth type} or just \textit{growth} of the function $\alpha$. For a
semigroup $S$ generated by a finite set $A$ the function
\[
g_S: \mathbb{N} \rightarrow \mathbb{N}, \quad g_S(m) = |\{  z \in S : l_A(z) \leq m  \}|
\]
is called the \emph{growth function} of the semigroup $S$ (where $l_A(z)$ denotes the length of the element $z$ as defined above). The growth \textit{type} of this function is independent of the choice of finite generating set and so is simply called the growth of the semigroup.  See for example
\cite{Shneerson2005} for more about growth of semigroups.

The following result follows from Definition~22 and Proposition~25 of \cite[\S VI]{delaHarpe2000}.

\begin{lemma}\label{growth_equiv_condition}
Let $\alpha_1$ and $\alpha_2$ be growth functions. 
Then $\alpha_1 \preccurlyeq \alpha_2$ if and only if there exist natural numbers $\lambda$ and $C$ such that
\[
\alpha_1(t) \leq \lambda \alpha_2(\lambda t + C) + C
\]
for all $t \in \mathbb{N}$.
\end{lemma}

\begin{definition}
A semimetric space $X$ is called \emph{uniformly quasi-locally bounded} if
\[
\sup_{x \in X} | \outball_t(x) | < \infty
\]
for all $t \in \mathbb{N}$. For such a space $X$, the \emph{growth function} of $X$ at $x_0 \in X$ 
is the function
$$\beta : \mathbb{N} \to \mathbb{N}, \quad t \mapsto |\outball_t(x_0)|.$$
\end{definition}

\begin{lemma}
Let $X$ be a semimetric space that is uniformly quasi-locally bounded, and
let $\beta_0$ and $\beta_1$ be the growth functions of $X$ at basepoints
$x_0$ and $x_1$ respectively. Then $\beta_0 \sim \beta_1$.
\end{lemma}
\begin{proof}
Since $x_0$ and $x_1$ are basepoints we have $d(x_0,x_1) < \infty$ and $d(x_1,x_0) < \infty$. Choose $C \in \mathbb{N}$ with $C > d(x_1,x_0)$. 
Now $\outball_t(x_0) \subseteq \outball_{t+C}(x_1)$ and hence $\beta_0(t) \leq \beta_1(t+C)$ for all $t \in \mathbb{N}$. 
Therefore $\beta_0 \preccurlyeq \beta_1$ by Lemma~\ref{growth_equiv_condition}. A dual argument establishes $\beta_1 \preccurlyeq \beta_0$.
\end{proof}

Let $S$ be a semigroup generated by a finite subset $A$.
Then the semimetric space $X = (S^1,d_A)$ is uniformly quasi-locally bounded and, with respect to the basepoint $1$, the growth function of $X$ is the growth function of the semigroup $S$. 

\begin{proposition}\label{prop_growth}
Let $X_1$ and $X_2$ be uniformly quasi-locally bounded
semimetric spaces with basepoints $x_1$ and $x_2$, respectively.
Let $\beta_j$ denote the corresponding growth function for $j=1,2$.
If $X_1$ quasi-isometrically embeds into $X_2$ then $\beta_1 \preccurlyeq \beta_2$. 
In particular, if $X_1$ and $X_2$ are quasi-isometric then they have the same type of growth.  
\end{proposition}
\begin{proof}
By Lemma~\ref{growth_equiv_condition} it will suffice to show that there exist natural numbers $\lambda$ and $C$ satisfying
$\beta_1 (t) \leq \lambda \beta_2 (\lambda t + C) + C$
for all $t \in \mathbb{N}$. 

Let $f: X_1 \rightarrow X_2$ be a $(\lambda,\epsilon)$-quasi-isometric
embedding. Clearly, we may assume without loss of generality that
$\lambda$ and $\epsilon$ lie in $\mathbb{N}$.
Since $x_2$ is a basepoint of $X_2$ it follows that $d_2(x_2, f(x_1)) < \infty$, and so we may choose a natural number $D$
with $D > d_2(x_2,f(x_1)) + \epsilon$. Now for any $x \in \outball_t(x_1)$
we have $d(x_1, x) \leq t$, and since is a $(\lambda,\epsilon)$-quasi-isometry
it follows that $d(f(x_1), f(x)) \leq \lambda t + \epsilon$ and hence
that
$$d(x_2, f(x)) \leq d(x_2, f(x_1)) + \lambda t + \epsilon \leq \lambda t + D$$
so that $f(x) \in \outball_{\lambda t + D}(x_2)$. Thus,
$f(\outball_t(x_1)) \subseteq \outball_{\lambda t + D}(x_2)$
and so
\[
|f(\outball_t(x_1))| \leq |\outball_{\lambda t + D}(x_2)| = \beta_2(\lambda t + D)
\]
for all $t \in \mathbb{N}$.
Next consider the fibres of the mapping $f$. For any $z \in X_2$ and any $x,y \in f^{-1}(z)$ we have $d_1(x,y) \leq \lambda \epsilon$. Since $X_1$ is uniformly quasi-locally bounded there is a natural number $E$ such that
\[
\sup_{z \in X_2} | f^{-1}(z) | \leq E.
\]
Therefore
 \[
|f(\outball(x_1,t))| \geq |\outball(x_1,t)| / E,
\]
and combining this with the previous formula we conclude that
\[
\beta_1(t) = |\outball(x_1,t)| \leq E |f(\outball(x_1,t))| \leq
E \beta_2(\lambda t + D)
\]
for all $t \in \mathbb{N}$. So setting $\mu = \max(E,\lambda) \in \mathbb{N}$ we obtain
\[
\beta_1(t) \leq \mu \beta_2(\mu t + D) \leq \mu \beta_2(\mu t + D) + D
\]
for all $t \in \mathbb{N}$, as required. 
\end{proof}

Let $S$ be a finitely generated semigroup. If $A$ and $B$ are two finite generating sets for $S$ then by Lemma~\ref{Independence} the semimetric spaces $(S,d_A)$ and $(S,d_B)$ are quasi-isometric, and so by Proposition~\ref{prop_growth}
the corresponding growth functions are of the same type. Thus (as mentioned above) the growth type depends only on the semigroup and not on the choice of generating set. More generally we have the following consequence of Proposition~\ref{prop_growth}.

\begin{theorem}\label{monoid_growth}
Let $S_1$ and $S_2$ be finitely generated semigroups with finite generating
sets $A_1$, $A_2$, and let $\beta_1$, $\beta_2$ be the corresponding growth
functions. If $(S_1,d_{A_1})$ and $(S_2,d_{A_2})$ quasi-isometrically embed
into each other then $S_1$ and $S_2$ have the same growth type. In
particular, growth type is a quasi-isometry invariant for finitely generated semigroups.
\end{theorem}

One application of Theorem~\ref{monoid_growth} is to provide examples of
pairs of semigroups which are \textit{not} quasi-isometric. For example
finitely generated free commutative monoids of different rank are not
quasi-isometric, since the free commutative monoid on $d$ generators
has polynomial growth rate of order $d$.

\section{Ends}\label{sec_ends}

In this section we study the relationship between quasi-isometry and the
number of \textit{ends} of a monoid.
The concept of ends goes back to work of Freudenthal \cite{Freudenthal1931, Freudenthal1942}, who introduced the notion to try to capture
the idea of connectivity at infinity for a topological space. Hopf \cite{Hopf1944} showed that 
if the translates of a compact set under the action of a group of homeomorphisms cover the whole space,
then the space has one, two or infinitely many ends (in fact, uncountably many). In particular a finitely generated group has one, two or infinitely many ends (since it acts transitively on its Cayley graph), and the celebrated result of Stallings \cite{Stallings1968} states that 
a finitely generated group has more than one end if and only if it
admits a nontrivial decomposition as an amalgamated free product or an HNN
extension over a finite subgroup. For finitely generated groups 
(and more generally for locally finite undirected graphs \cite[Proposition~1]{Moller1992})
it is well-known that the number of ends is a quasi-isometry invariant. See
\cite{DicksAndDunwoody} for more on the importance of ends in group theory. 

Jackson and Kilibarda \cite{JacksonAndKilibarda} have recently initiated the study of ends
of finitely generated monoids. They define the number of ends of a monoid
$M$ to be the supremum number of infinite connected components that may be
obtained by removing a finite set of vertices from the underlying undirected
graph of the right Cayley graph of $M$. 
In this section we show that the
number of ends of a monoid, defined in this way, is a quasi-isometry
invariant. We first need a preliminary lemma.

\begin{lemma}\label{lemma_one}
Let $f : M \to N$ be a quasi-isometric embedding of monoids and
$S \subseteq M$ a subset of $M$. If $S$ is infinite then $f(S)$ is infinite.
\end{lemma}
\begin{proof}
Suppose false for a contradiction. Then by the pigeon hole principle,
there exists an infinite subset $T \subseteq S$ and a
point $n \in N$ such that $f(T) = \lbrace n \rbrace$. Now since the
Cayley graph of $M$ has finite outdegree, $T$ must contain elements $t$
such that $d(1_M, t)$ is finite but arbitrarily large. But
$d(f(1_M), f(t)) = d(f(1_M), n)$ is constant as $t$ varies within $T$,
which clearly contradicts the assumption that $f$ is a quasi-isometric
embedding.
\end{proof}

\begin{theorem}\label{thm_ends}
Let $(M, X)$ and $(N,Y)$ be quasi-isometric finitely generated
monoids. Then $(M,X)$ and $(N,Y)$ have the same number of ends.
\end{theorem}
\begin{proof}
Let $f : M \to N$ be a $(\lambda,\epsilon,\mu)$-quasi-isometry.

Suppose $S \subseteq M$ is a finite subset separating the Cayley graph of
$M$ into $r$ infinite components $C_1, \dots, C_r$. By
\cite[Lemma~5]{JacksonAndKilibarda}, $S$ separates $M$ into only finitely
many components in total. It follows that we
may absorb any finite components into $S$, and assume without loss of
generality that
$$M = S \cup C_1 \cup \dots \cup C_r.$$
Let $\omega = \lambda^2 (2 \mu + \epsilon + 1) + \epsilon$
and define
$$T = \lbrace t \in N \mid d(f(s), t) \leq \omega \textrm{ for some } s \in S \rbrace \text{ and }$$
$$U = \lbrace t \in N \mid d(f(s), t) \leq \omega + \mu \textrm{ for some } s \in S \rbrace.$$
Notice that $T$ and $U$ are finite, because $S$ is finite and the outdegree of each vertex in the Cayley
graph of $N$ is finite.
For $1 \leq i \leq r$ define $D_i = f(C_i) \setminus T$, and $E_i = f(C_i) \setminus U$.
We claim that $U$ separates $N$ into components, at least $r$ of which are
infinite.

We first claim that if $p \neq q$, there are no directed paths of length less
than $2 \mu + 1$ from $D_p$ to $D_q$. Indeed, suppose for a contradiction
that there was such a path. Since $D_p \subseteq f(C_p)$ and
$D_q \subseteq f(C_q)$ we may assume the path runs from $f(r) \in D_p$
to $f(s) \in D_q$ where $r \in C_p$ and $s \in C_q$. Now
since $f$ is a $(\lambda,\epsilon,\mu)$-quasi-isometry, there is a directed
path of length at most
$\lambda (2 \mu + 1 + \epsilon)$ from $r$ to $s$. By the assumption
on $S$, this path must pass through $S$; let $t \in S$ be a vertex on it.
Then there is a directed path from $t$ to $s$ of length at most
$\lambda (2 \mu + 1 + \epsilon)$, so using once again the fact that $f$ is
a $(\lambda,\epsilon,\mu)$-quasi-isometry, there is a directed path from
$f(t) \in f(S)$ to $f(s)$ of length at most $\lambda [\lambda (2 \mu + 1 + \epsilon)] + \epsilon = \omega$.
But by the definition of $T$, this means that $f(s) \in T$, which contradicts
the claim that $f(s) \in D_q = C_q \setminus T$.

Next, we claim that for $1 \leq i < j \leq k$, every
undirected path connecting a vertex in $D_i$ to a vertex in $D_j$ passes through $U$. Indeed,
suppose for a contradiction that $\pi$ is an undirected path from some $D_i$ to
some $D_j$ which does \textit{not} pass through $U$. We shall show that
there is a directed path of length at most $2\mu+1$
from some $D_p$ to some $D_q$ (with $p \neq q$), thus contradicting the
previous claim. Since $\pi$ is a path between vertices in a graph,
we may assume it has integer length. First, if $\pi$ has length
$1$ then it is a directed path (either from $D_i$ to $D_j$ or from $D_j$ to
$D_i$) and so clearly has the required form. Next, if the final vertex of $\pi$
lies within a strong ball of radius $\mu$ around $D_i$, then clearly there
is a path of length $\mu$ from $D_i$ to $D_j$ so the claim again holds.

Otherwise, let $x$
be the first vertex on the path $\pi$ which does \textit{not} lie within a
strong ball of
radius $\mu$ around $D_i$, and let $w$ be the immediately
preceding vertex. Then there is either an edge from $x$ to $w$, or an
edge from $w$ to $x$. Since $f$ is a $(\lambda,\epsilon,\mu)$-quasi-isometry,
$f(M)$ is $\mu$-quasidense in $N$, so $x$ lies within a strong ball
of radius $\mu$ around $f(y)$ for some $y \in M$. Since
$M = S \cup C_1 \cup \dots \cup C_r$ we must have $y \in S$ or $y \in C_k$
for some $k$. If $y$ were in $S$ then, since $\mu \leq \omega$, by the
definition of $T$ we would have
$x \in T \subseteq U$, contradicting the assumption on $\pi$. So instead we
must have $y \in C_k$ for some $k$. Now if $f(y) \in T$ then since
$d(f(y),x) \leq \mu$ we would again have $x \in U$, once more contradicting
the assumption on $\pi$. There remains only the possibility that
$f(y) \in D_k$. Moreover, since $x$ was chosen not to be within a strong
ball of radius $\mu$ around $d_i$, we have $k \neq i$.

Now suppose there is an edge from $x$ to $w$. Since $x$ and $w$ are
within strong balls of radius $\mu$ around $D_i$ and $D_k$
respectively, there is a directed path of length at most $2 \mu + 1$ from
$D_i$ to $D_k$, which is a contradiction. On the other hand, if there is an
edge from $w$ to $x$ then by the same argument, there is a directed path of
length at most $2 \mu + 1$ from $D_k$ to $D_i$, which is again a contradiction.

We have now shown that every undirected path connecting $D_i$ to $D_j$
with $i \neq j$ passes through $U$. In particular, since $E_i \subseteq D_i$
for each $i$, every undirected path from $E_i$ to $E_j$ passes through $U$.
To complete the proof that $T$ separates $N$ into at least $r$ components, it
will suffice to show that each $E_i$ contains an infinite set of vertices
which remain connected when $U$ is removed.
By \cite[Lemma~5]{JacksonAndKilibarda} again, $U$ separates $N$ into only finitely
many components. Since it separates $E_i$ from everything else, it follows
that it separates $E_i$ into only finitely many components. But $C_i$ is
infinite so Lemma~\ref{lemma_one} ensures that $f(C_i)$ is infinite, and
since $U$ is finite it follows that $E_i = f(C_i) \setminus U$ is infinite,
as required.

We have now shown that if $M$ can be separated into $r$ components, then
$N$ can be separated into at least $r$ components. It follows that if $M$
has $r$ ends then $N$ has at least $r$ ends, while if $M$ has infinitely
many ends then so does $N$. By symmetry of assumption, it follows that
$M$ and $N$ have the same number of ends.
\end{proof}

Combining Theorem~\ref{thm_ends} with Proposition~\ref{Independence}, we
recover the following important
fact, which was first proved in \cite{JacksonAndKilibarda} by an entirely
different combinatorial argument.
\begin{corollary}
The number of ends of a finitely generated monoid is independent of the
choice of finite generating set.
\end{corollary}

As mentioned above, it is well-known that every Cayley graph of a finitely generated group has
one, two or infinitely many ends. Jackson and
Kilibarda \cite{JacksonAndKilibarda} showed the corresponding statement for monoids does
not hold. However, as a corollary of Theorem~\ref{quasiisomschutz}, we
see that in the case of $\gr$-classes with finitely many $\gh$-classes,
a corresponding statement does hold for Sch\"{u}tzenberger graphs.

\begin{corollary}
Let $M$ be a finitely generated monoid and let $R$ be an $\gr$-class of $M$.
If $R$ is a union of finitely many $\gh$-classes then the Sch\"{u}tzenberger
graph of $R$ has $1$, $2$ or infinitely many ends.
\end{corollary}

\section{Examples}

In this section we give some examples of monoids that are quasi-isometric to one another.
The idea roughly speaking is that for two monoids to be
quasi-isometric, globally they must share the same $\gr$-class poset structure, and locally they must have group structure that is the same up to quasi-isometry. Recall that a semigroup is called \emph{right simple} if it has a single $\gr$-class. 

\begin{proposition}
The property of being a group is a quasi-isometry invariant of
finitely generated monoids.
\end{proposition}
\begin{proof}
Let $M$ be a finitely generated monoid that is quasi-isometric to a finitely generated group $G$.
Since the number of $\gr$-classes is a quasi-isometry invariant, this implies that $M$ is right simple (i.e. has a single $\gr$-class). But any right simple monoid is necessarily a group (see \cite[Theorem~1.27]{CliffordAndPreston}).
\end{proof}

\begin{proposition}
A finitely generated semigroup is right simple if and only if it is quasi-metric.
\end{proposition}
\begin{proof}
Let $S$ be a semigroup generated by a finite set $A \subseteq S$. If $S$ is quasi-metric then $(S,d_A)$ is strongly connected and hence is right simple. 

Conversely suppose that $S$ is right simple. Define 
\[
\lambda = \max \{ d_A(ba,a) : a,b \in A  \}
\]
which exists since $A$ is finite, and is finite since $S$ is right simple. Then for all $x \in S = \langle A \rangle$ and $a \in A$, writing $x = x'b$ where $b \in A$, we obtain
\[
d_A(xa,x) \leq d_A(ba,b) \leq \lambda.
\]
It follows that for all $x,y \in S$
\[
d(y,x) \leq \lambda d(x,y)
\]
and therefore $(S,d_A)$ is quasi-metric. 
\end{proof}

It is well-known, and follows quite easily from the \v{S}varc--Milnor Lemma,
that a surjective morphism of finitely generated groups is a quasi-isometry
if and only if it has finite kernel. The natural extension of this fact to
monoids is the following.
\begin{proposition}
Let $S$ be a semigroup generated by a finite subset $A$, and let $\eta$ be a
congruence on $S$. Then the natural map $\phi: S \rightarrow S/\eta$
sending $s \mapsto s / \eta$ is a quasi-isometry if and only
if there is a bound on the $d_A$-diameter of the $\eta$-classes of $S$.
\end{proposition}
\begin{proof}
First suppose that there is a bound, say $R>0$, on the $d_A$-diameter of the $\eta$-classes of $S$. We claim that the map $\phi:S \rightarrow S / \eta$ given by $s \mapsto s / \eta$ is a $(\lambda, \epsilon, \mu)$
quasi-isometry where $\lambda=1$, $\epsilon=R$ and $\mu=0$.

The mapping $\phi$ is surjective and hence $\phi(S)$ is $0$-quasi-dense in $S / \eta$.

Let $x,y \in S$. If $d_A(x,y)=r$ then we can write $x a_1 \cdots a_r$ where $a_i \in A$ for all $i$, and so $$(x/\eta) (a_1/\eta) \cdots (a_r / \eta) = y / \eta$$ which implies $d_{A / \eta} (\phi(x),\phi(y)) \leq r = d(x,y)$.

Also, if $d_{A/\eta}(f(x),f(y)) = k$ then this means that we can write $$(x/\eta)(a_1/\eta) \cdots (a_k/\eta) = (y/\eta),$$ where $a_i \in A$ for all $i$. This implies $(xa_1 \cdots a_k,y) \in \eta$ and so $d_A(x,y) \leq k + R = d_{A/\eta}(f(x),f(y)) + R$. We conclude that for all $x,y \in S$
\[
d_A(x,y) - R \leq d_{A/\eta}(f(x),f(y)) \leq d_A(x,y),
\]
as required.

For the converse suppose that $\phi$ is a quasi-isometry, say with constants $(\lambda, \epsilon, \mu)$. Let $x,y \in S$ with $(x,y) \in \eta$. Then $\phi(x) = \phi(y)$, so
$$\frac{1}{\lambda} d_A(x,y) - \epsilon \leq d_{A/\eta}(\phi(x),\phi(y))=0$$
which implies $d_A(x,y) \leq \epsilon \lambda$. Thus $\epsilon \lambda$ is an upper bound on the $d_A$-diameter of the $\eta$-classes of $S$.
\end{proof}

Note that it does \textit{not} suffice for the above result to require that
the \textit{cardinalities} of $\eta$-classes be bounded, let alone merely
that they be finite. For example if $S$ is any finite monoid with more than
one $\gr$-class then the homomorphism from $S$ onto the trivial monoid satisfies
these conditions, but $S$ is not quasi-isometric to the trivial monoid.

A particular instance where the conditions of the above proposition are satisfied is when
$S \cong M \times G$ where $G$ is a finite group, $M$ is a monoid 
and $\eta$ is the congruence corresponding to the natural projection $M \times G \rightarrow M$, $(m,g) \mapsto m$.

\begin{corollary}\label{corol_direct_product}
Let $M$ be a finitely generated monoid, and let $G$ be a finite group.
Then $M \times G$ is quasi-isometric to $M$.
\end{corollary}

In fact, for the same reason this result holds more generally for the semidirect product of a monoid and a group (with the monoid acting on the group by endomorphisms). 

If $G$ is a finitely generated group and $H$ is a subgroup of $G$ of finite index then, as a standard application of the \v{S}varc--Milnor lemma, $H$ is finitely generated
and is quasi-isometric to $G$. This fact may be used to give further examples of quasi-isometric monoids. 
We use $E(S)$ to denote the set of idempotents of a semigroup $S$. 

\begin{proposition}
Let $M$ and $N$ be finitely generated Clifford monoids and let $\phi: M \rightarrow N$ be an idempotent separating homomorphism such that $E(N) \subseteq \phi(M)$. If for every $e \in E(N)$ the pre-image $\phi^{-1}(e)$ has finite index in the maximal subgroup of $M$ that contains it, then $M$ and $N$ are quasi-isometric.
\end{proposition}

For similar reasons we have the following observation.

\begin{proposition}
Let $S = M^0[G;I, \Lambda, P]$ be a finitely generated completely $0$-simple semigroup represented as a $0$-Rees matrix semigroup over a group. Let $H$ be a subgroup of $G$ and suppose that every non-zero entry of $P$ belongs to $H$. If $H$ has finite index in $G$ then $S$ and $T = M^0[H;I,\Lambda,P]$ are quasi-isometric.
\end{proposition}

\section*{Acknowledgements}

The research of the first author is supported by an EPSRC Postdoctoral Fellowship. The
research of the second author is supported by an RCUK Academic Fellowship. The
second author gratefully acknowledges the financial support and hospitality
of the \textit{Centre for Interdisciplinary Research in Computational Algebra} during a visit to
St Andrews.


\begin{thebibliography}{99}

\bibitem{Arbib1968}
{\em Algebraic theory of machines, languages, and semigroups}.
\newblock Edited by M.~A.~Arbib. With a major contribution by K.~Krohn and J.~L.~Rhodes. Academic Press, New York, 1968.

\bibitem{Alonso1990}
J.~M. Alonso.
\newblock In\'egalit\'es isop\'erim\'etriques et quasi-isom\'etries.
\newblock {\em C. R. Acad. Sci. Paris S\'er. I Math.}, 311(12):761--764, 1990.

\bibitem{Alonso1994}
J.~M. Alonso.
\newblock Finiteness conditions on groups and quasi-isometries.
\newblock {\em J. Pure Appl. Algebra}, 95(2):121--129, 1994.

\bibitem{Bergman1978}
G.~M. Bergman.
\newblock A note on growth functions of algebras and semigroups.
\newblock {\em Technical report, Department of Mathematics, University of
  California, Berkeley}, 1978.

\bibitem{Bridson1996}
M.~R. Bridson and S.~M. Gersten.
\newblock The optimal isoperimetric inequality for torus bundles over the
  circle.
\newblock {\em Quart. J. Math. Oxford Ser. (2)}, 47(185):1--23, 1996.

\bibitem{CliffordAndPreston}
A.~H. Clifford and G.~B. Preston
\newblock {\em The algebraic theory of semigroups. {V}ol. {I}}, {\em Mathematical Surveys, No. 7}.
\newblock  American Mathematical Society, 
\newblock  Providence, R.I. 1961. 

\bibitem{Cremanns1996}
R.~Cremanns and F.~Otto.
\newblock For groups the property of having finite derivation type is
  equivalent to the homological finiteness condition {${\rm FP}\sb 3$}.
\newblock {\em J. Symbolic Comput.}, 22(2):155--177, 1996.

\bibitem{delaHarpe2000}
P.~de~la Harpe.
\newblock {\em Topics in geometric group theory}.
\newblock Chicago Lectures in Mathematics. University of Chicago Press,
  Chicago, IL, 2000.

\bibitem{DicksAndDunwoody}
W.~Dicks and M.~J. Dunwoody.
\newblock {\em Groups acting on graphs}, volume~17 of {\em Cambridge Studies in
  Advanced Mathematics}.
\newblock Cambridge University Press, Cambridge, 1989.

\bibitem{EilenbergA}
S.~Eilenberg.
\newblock {\em Automata, languages, and machines. {V}ol. {A}}.
\newblock Academic Press [A subsidiary of Harcourt Brace Jovanovich,
  Publishers], New York, 1974.
\newblock Pure and Applied Mathematics, Vol. 58.

\bibitem{EilenbergB}
S.~Eilenberg.
\newblock {\em Automata, languages, and machines. {V}ol. {B}}.
\newblock Academic Press [Harcourt Brace Jovanovich Publishers], New York,
  1976.
\newblock With two chapters (``Depth decomposition theorem'' and ``Complexity
  of semigroups and morphisms'') by Bret Tilson, Pure and Applied Mathematics,
  Vol. 59.

\bibitem{Freudenthal1931}
H.~Freudenthal.
\newblock \"{U}ber die {E}nden topologischer {R}\"aume und {G}ruppen.
\newblock {\em Math. Z.}, 33(1):692--713, 1931.

\bibitem{Freudenthal1942}
H.~Freudenthal.
\newblock Neuaufbau der {E}ndentheorie.
\newblock {\em Ann. of Math. (2)}, 43:261--279, 1942.

\bibitem{Ghys1991}
{\'E}.~Ghys and P.~de~la Harpe.
\newblock Infinite groups as geometric objects (after {G}romov).
\newblock In {\em Ergodic theory, symbolic dynamics, and hyperbolic spaces
  (Trieste, 1989)}, Oxford Sci. Publ., pages 299--314. Oxford Univ. Press, New
  York, 1991.

\bibitem{GrayMalheiro}
R.~Gray and A.~Malheiro.
\newblock Homotopy bases and finite derivation type for subgroups of monoids.
\newblock {\em submitted} (available at 
http://www-history.mcs.st-and.ac.uk/\~{}robertg/preprints.html).


\bibitem{Gray2008}
R.~Gray and N.~Ru{\v{s}}kuc.
\newblock Green index and finiteness conditions for semigroups.
\newblock {\em J. Algebra}, 320(8):3145--3164, 2008.

\bibitem{Green1951}
J.~A. Green.
\newblock On the structure of semigroups.
\newblock {\em Ann. of Math. (2)}, 54:163--172, 1951.

\bibitem{Grigorchuk1988}
R.~I. Grigorchuk.
\newblock Semigroups with cancellations of degree growth.
\newblock {\em Mat. Zametki}, 43(3):305--319, 428, 1988.

\bibitem{Gromov1981}
M.~Gromov.
\newblock Groups of polynomial growth and expanding maps.
\newblock {\em Inst. Hautes \'Etudes Sci. Publ. Math.}, (53):53--73, 1981.

\bibitem{Guba1997}
V.~Guba and M.~Sapir.
\newblock Diagram groups.
\newblock {\em Mem. Amer. Math. Soc.}, 130(620):viii+117, 1997.

\bibitem{Hopf1944}
H.~Hopf.
\newblock Enden offener {R}\"aume und unendliche diskontinuierliche {G}ruppen.
\newblock {\em Comment. Math. Helv.}, 16:81--100, 1944.

\bibitem{Howie1995}
J.~M. Howie.
\newblock {\em Fundamentals of semigroup theory}, volume~12 of {\em London
  Mathematical Society Monographs. New Series}.
\newblock The Clarendon Press Oxford University Press, New York, 1995.
\newblock Oxford Science Publications.

\bibitem{JacksonAndKilibarda}
D.~A. Jackson and V.~Kilibarda.
\newblock Ends for monoids and semigroups.
\newblock {\em J. Aust. Math. Soc.}, to appear.

\bibitem{Kelly1963}
J.~C. Kelly.
\newblock Bitopological spaces.
\newblock {\em Proc. London Math. Soc. (3)}, 13:71--89, 1963.

\bibitem{Lallement1979}
G.~Lallement.
\newblock {\em Semigroups and combinatorial applications}.
\newblock John Wiley \& Sons, New York-Chichester-Brisbane, 1979.
\newblock Pure and Applied Mathematics, A Wiley-Interscience Publication.

\bibitem{Milnor1968}
J.~Milnor.
\newblock A note on curvature and fundamental group.
\newblock {\em J. Differential Geometry}, 2:1--7, 1968.

\bibitem{Moller1992}
R.~G. M{\"o}ller.
\newblock Ends of graphs. {II}.
\newblock {\em Math. Proc. Cambridge Philos. Soc.}, 111(3):455--460, 1992.

\bibitem{Otto1997}
F.~Otto and Y.~Kobayashi.
\newblock Properties of monoids that are presented by finite convergent
  string-rewriting systems---a survey.
\newblock In {\em Advances in algorithms, languages, and complexity}, pages
  225--266. Kluwer Acad. Publ., Dordrecht, 1997.

\bibitem{Pansu1983}
P.~Pansu.
\newblock Croissance des boules et des g\'eod\'esiques ferm\'ees dans les
  nilvari\'et\'es.
\newblock {\em Ergodic Theory Dynam. Systems}, 3(3):415--445, 1983.

\bibitem{Richter-Gebert2005}
J.~Richter-Gebert, B.~Sturmfels, and T.~Theobald.
\newblock First steps in tropical geometry.
\newblock In {\em Idempotent mathematics and mathematical physics}, volume 377
  of {\em Contemp. Math.}, pages 289--317. Amer. Math. Soc., Providence, RI,
  2005.

\bibitem{Ruskuc2000}
N.~Ru{\v{s}}kuc.
\newblock On finite presentability of monoids and their {S}ch\"utzenberger
  groups.
\newblock {\em Pacific J. Math.}, 195(2):487--509, 2000.

\bibitem{Schutzenberger1957}
M.~P. Sch{\"u}tzenberger.
\newblock {$\mathcal{D}$}-repr\'esentation des demi-groupes.
\newblock {\em C. R. Acad. Sci. Paris}, 244:1994--1996, 1957.

\bibitem{Schutzenberger1958}
M.~P. Sch{\"u}tzenberger.
\newblock Sur la repr\'esentation monomiale des demi-groupes.
\newblock {\em C. R. Acad. Sci. Paris}, 246:865--867, 1958.

\bibitem{Shneerson2001}
L.~M. Shneerson.
\newblock Relatively free semigroups of intermediate growth.
\newblock {\em J. Algebra}, 235(2):484--546, 2001.

\bibitem{Shneerson2005}
L.~M. Shneerson.
\newblock Types of growth and identities of semigroups.
\newblock {\em Internat. J. Algebra Comput.}, 15(5-6):1189--1204, 2005.

\bibitem{Shneerson2008}
L.~M. Shneerson.
\newblock Polynomial growth in semigroup varieties.
\newblock {\em J. Algebra}, 320(6):2218--2279, 2008.

\bibitem{Stallings1968}
J.~R. Stallings.
\newblock On torsion-free groups with infinitely many ends.
\newblock {\em Ann. of Math. (2)}, 88:312--334, 1968.

\bibitem{Steinberg2001}
B.~Steinberg.
\newblock Finite state automata: a geometric approach.
\newblock {\em Trans. Amer. Math. Soc.}, 353(9):3409--3464 (electronic), 2001.

\bibitem{Steinberg2003}
B.~Steinberg.
\newblock A topological approach to inverse and regular semigroups.
\newblock {\em Pacific J. Math.}, 208(2):367--396, 2003.

\bibitem{Steinberg2009}
B.~Steinberg.
\newblock A groupoid approach to discrete inverse semigroup algebras.
\newblock arXiv:0903.3456v1, 2009.

\bibitem{Svarc1955}
A.~S. {\v{S}}varc.
\newblock A volume invariant of coverings.
\newblock {\em Dokl. Akad. Nauk SSSR (N.S.)}, 105:32--34, 1955.

\bibitem{Wilson1931}
W.~A. Wilson.
\newblock On {Q}uasi-{M}etric {S}paces.
\newblock {\em Amer. J. Math.}, 53(3):675--684, 1931.

\end{thebibliography}
\end{document}